\documentclass[12pt]{amsart}
\usepackage{amsmath,amsthm,amsfonts,amssymb,latexsym,enumerate,xcolor}
\usepackage{hyperref}
\usepackage[english]{babel}
\usepackage[capitalize, noabbrev]{cleveref}

\newcommand{\Q}{{\mathbb{Q}}}

\newcommand{\N}{{\mathbb{N}}}

\newcommand{\F}{\mathbb{F}}

\newcommand{\Aut}{{{\operatorname{Aut}}}}

\newcommand{\Irr}{{{\operatorname{Irr}}}}

\newcommand{\Cl}{\operatorname{Cl}}
\newcommand{\Exp}{\operatorname{Exp}}

\newcommand{\Gal}{\operatorname{Gal}}

\newtheorem{thm}{Theorem}[section]
\newtheorem{lem}[thm]{Lemma}
\newtheorem{con}[thm]{Conjecture}
\newtheorem{pro}[thm]{Proposition}
\newtheorem{cor}[thm]{Corollary}

\newtheorem*{conA'}{Conjecture A'}

\newtheorem{thml}{Theorem}

\theoremstyle{definition}

\numberwithin{equation}{section}

\author{Juan Mart\'inez Madrid}
\address{Departament de Matem\`atiques, Universitat de Val\`encia, 46100
  Burjassot, Val\`encia, Spain}
\email{Juan.Martinez-Madrid@uv.es}

\author{Marco Vergani}
\address{Dipartimento di Matematica e Informatica U. Dini,\newline
Universit\`a degli Studi di Firenze, viale Morgagni 67/a,
50134 Firenze, Italy.}
\email{marco.vergani@unifi.it}

\thanks{The first author was  supported by Ministerio de Ciencia e Innovaci\'on (Grant PID2022-137612NB-I00 funded by MCIN/AEI/10.13039/501100011033 and ``ERDF A way of making Europe"),   and by Generalitat Valenciana  CIACIF/2021/228. 
The second author is partially supported by INdAM-GNSAGA. This research is also funded by the European Union-Next Generation EU, Missione 4 Componente 1, CUP B53D23009410006, PRIN 2022 2022PSTWLB - Group Theory and Applications.}

\begin{document}


\title[Multiplicities of fields of values of conjugacy classes]{Multiplicities of fields of values of conjugacy classes in finite groups}




\begin{abstract}
Given a group $G$ we write $h(G)$ to denote the maximum  number of times that a field extension of $\Q$ appears as the field of values of a conjugacy class of a group. In this work, we prove that $|G|$ is bounded in terms of $h(G)$. Moreover, we classify  the groups with $h(G)\leq 3$. 
\end{abstract}

\maketitle

\pagestyle{myheadings}



\section{Introduction}\label{Section1}

All throughout this paper $G$ is a finite group. We will write $\Cl(G)$ and $\Irr(G)$ to denote the set of conjugacy classes and the set of complex irreducible characters, respectively. If $\chi$ is a character of $G$ and $K\in \Cl(G)$, then we define $\chi(K)=\chi(g)$ for any $g \in K$.

Given a character $\chi$ of $G$ (not necessary irreducible), we can associate a finite extension of $\Q$ to $\chi$, namely the field generated by the values $\chi(K)$ for $K$ in $\Cl(G)$. We denote this extension by

\[\Q(\chi)=\Q(\chi(K)\mid K \in \Cl(G)).\]  

The extension $\Q(\chi)$ is called \textbf{the field of values} of $\chi$.

Analogously, given $K \in \Cl(G)$, we define the \textbf{field of values} of $K$ as the field extension of $\Q$ generated by the values $\chi(K)$, where $\chi$ runs through $\Irr(G)$. That is
\[\Q(K)=\Q(\chi(K)\mid\chi \in \Irr(G)).\] 
Let $K \in \Cl(G)$ and $g \in K$. Let $|g|$ denote the order of $g$. We define $\Q(g)=\Q(K)$. We know that $\chi(g)$ can be expressed as the sum of $|g|$-th roots of the unity for any $\chi \in \Irr(G)$. It follows that $\Q(K)=\Q(g)\subseteq  \Q_{|g|}$, where $\Q_{|g|}=\Q(e^{\frac{2i\pi}{|g|}})$.

Finally, we define the field of values of $G$ as the field generated by the values $\chi(K)$ where $\chi$ runs through $\Irr(G)$ and $K $ runs through $\Cl(G)$. That is
\[\Q(G)=\Q(\chi(K)\mid\chi \in \Irr(G), K \in \Cl(G)).\]
By the previous observation, we have that  $\Q(G)\subseteq \Q_{\Exp(G)}$, where $\Exp(G)$ denotes the exponent of $G$.

Nowadays, the study of fields of values is a very active area of research in character theory.  Given a group $G$, we write $\pi(G)$ to denote the set of prime divisors of $|G|$. In the case of solvable groups, it is possible to bound the prime divisors of $|G|$ in terms of fields of values of classes and irreducible characters. A classical result of Gow \cite{G} shows that $\pi(G)\subseteq  \{2,3,5\}$ for any solvable group $G$ with $\Q(G)=\Q$. Chillag and Dolfi \cite{CD} proved that if $G$ is a solvable group such that $\Q(K)$ is either $\Q$ or a quadratic extension of $\Q$ for every $K \in \Cl(G)$, then $\pi(G)\subseteq \{2,3,5,7,13,17\}$. It is not known whether $17$ can actually occur as a prime divisor of a group that satisfies this property. Moreover, Tent \cite{Tent} proved a dual version for characters. More precisely, Tent showed that if $G$ is a solvable group such that $\Q(\chi)$ is either $\Q$ or a quadratic extension of $\Q$ for every $\chi \in \Irr(G)$, then $\pi(G)\subseteq \{2,3,5,7,13\}$.

It is worth mentioning that there exists a common ingredient in the proofs of many of the  results in \cite{CD,G,Tent}. In many cases, they use the techniques developed by Farias e Soares \cite{farias}, especially the concept of $h$-eigenvalue property.

Given a group $G$, we write $k(G)$ to denote the number of conjugacy classes of $G$. A classical result of Landau \cite{Landau} shows that $|G|$ is bounded in terms of $k(G)$. That is, there exists a non-decreasing function $a:\N \to \N$ such that $|G|\leq a(k(G))$ for any  group $G$.

There are many generalizations of Landau's Theorem (see, for example, \cite{HK,mg}). In this work, we are interested in bounding $|G|$ in terms of invariants defined on the fields of values. Given a group $G$,  Moretó \cite{Alex} defined the invariant $f(G)$ as
\[f(G)=\max_{F/\mathbb{Q}}|\{\chi \in \Irr(G)\mid\mathbb{Q}(\chi)=F\}|.\]

Theorem A of \cite{Alex} shows that $|G|$ is bounded in terms of $f(G)$. The first goal this work is to prove a dual version of that result for fields of values of conjugacy classes. For a group $G$, we define the invariant $h(G)$ by
\[h(G)=\max_{F/\mathbb{Q}}|\{K\in \Cl(G)\mid\mathbb{Q}(K)=F\}|.\]
With this definition, we prove an analogous version of Theorem A of \cite{Alex} for conjugacy classes.

\begin{thml}\label{thmA}
     Let $G$ be a finite group. Then $|G|$ is $\hat{h}(G)$-bounded.
\end{thml}

It is important to remark that Theorem \ref{thmA} is a consequence of a generalization of Landau's Theorem, which was recently proved by    Çinarci,   Keller,  Maróti,   Simion \cite{CKMS}.

On the other hand, Moretó \cite{Alex} observed that $f(G)=1$ if and only if $G=1$. The same paper asks for the classification of all groups with $f(G)\in \{2,3\}$. This classification was obtained in \cite{J}. By inspection of the groups appearing in that classification, we observe that if $f(G)\in \{1,2,3\}$, then $h(G)=f(G)$. So we asked if $f(G)=h(G)$ for all groups with $h(G)\leq 3$. Our second main result answers that question affirmatively.
 
 \begin{thml}\label{thmB}
 Let $G$ be a finite group. If $h(G)\leq 3$, then $f(G)=h(G)$.
 \end{thml}

It is easy to see that if $G$ is a group with $h(G)\leq 3$, then $G$  has at most three rational (conjugacy) classes. Then applying the main results developed by Navarro, Rossi and Tiep \cite{NT,Rossi}, we deduce that the number of rational irreducible character in $G$ is equal to the number of rational classes in $G$. Theorem \ref{thmB} will be proved by using this fact, together with a careful application of Brauer's Permutation Lemma.

We prove Theorem \ref{thmA} in Section  \ref{Section2} and Theorem \ref{thmB} in Section \ref{SecB}.

\section{Proof of Theorem \ref{thmA}}\label{Section2}

We begin by introducing the results on the fields of values of characters and conjugacy classes. This is deeply related with  the action of the Galois group on conjugacy classes and on irreducible characters.

Given $\chi \in \Irr(G)$ and $\sigma \in \Gal(\Q_{|G|}/\Q)$ we define $\chi^{\sigma}$ as $\chi^{\sigma}(g)=\chi(g)$ for any $g \in G$ (we recall that $\Q(\chi)\subseteq \Q_{|G|}$). This defines an action of $\Gal(\Q_{|G|}/\Q)$ acts  on $\Irr(G)$.

  In addition, $\Gal(\Q_{|G|}/\Q)$ acts  on $\Cl(G)$. Let $\epsilon$ be a $|G|$-th primitive root of unity. Then each $\sigma \in \Gal(\Q_{|G|}/\Q)$  is determined by the unique $r$ such that $0<r<|G|$, $\gcd(r,|G|)=1$ and $\sigma(\epsilon)=\epsilon^r$. Therefore, $\Gal(\Q_{|G|}/\Q)$ acts on $G$ and on $\Cl(G)$ as $ g^\sigma=g^r$ and $ (g^G)^\sigma=(g^r)^G$. Using this action we can prove the following.

\begin{lem}\label{Field Cont}
Let  $g \in G$  and let $j$ be a positive integer. Then $\Q(g^j)\subseteq \Q(g)$.
\begin{proof}
Let $n$ be the exponent of $G$ and let $\epsilon$ be a primitive $n$-th root of unity. We claim that  $\Gal(\Q_n/\Q(g))\leq \Gal(\Q_n/\Q(g^j))$. Let $\sigma\in \Gal(\Q_n/\Q(g))$. Assume that  $\sigma(\epsilon)=\epsilon^r$ for $0<r<|G|$ and  $\gcd(r,|G|)=1$. Since $\sigma\in \Gal(\Q_n/\Q(g))$, we have that $(g^G)^{\sigma}=g^G$. Thus, $g^r$ is conjugate to $g$. It follows that $(g^j)^r$ is conjugate to $g^j$ and hence $\sigma\in \Gal(\Q_n/\Q(g^j))$. The claim follows.

Therefore, $\Gal(\Q_n/\Q(g))\leq \Gal(\Q_n/\Q(g^j))$ and hence $\Q(g^j)\subseteq \Q(g)$.
\end{proof}
\end{lem}

Now, we relate $h(G)$ with the action of $\Gal(\Q_{|G|}/\Q)$ on $\Cl(G)$.

\begin{lem}\label{krational}
Let  $g \in G$. Then $|\Q(g):\Q|\leq h(G)$.
\begin{proof}
Let $\sigma \in \Gal(\Q(g)/\Q)\setminus\{1\}$. We can extend $\sigma$ to an element in $\Gal(\Q_{|G|}/\Q)$. Then $\Q(g)=\Q(g^{\sigma})$ and $g^G\neq (g^G)^{\sigma}$. It follows that there exist at least $|\Gal(\Q(g)/\Q)|$ different conjugacy classes with the same field of values. It follows that $|\Q(g):\Q|=|\Gal(\Q(g)/\Q)|\leq h(G)$.
\end{proof}
\end{lem}

Lemma \ref{krational} leads us to give the following definition. Let $G$ be a finite group and let $k\geq 1$ be an integer, we say that $g \in G$  is \textbf{$\frac{1}{k}$-rational} if $|\Q(g):\Q|\leq k$.
If $k=1$, then $g$ is said to be \textbf{rational} and if $k=2$, then $g$ is said to be \textbf{semi-rational}. The group $G$ is said to be $\frac{1}{k}$-rational (resp. rational or semi-rational) if every element is $\frac{1}{k}$-rational (resp. rational or semi-rational). Thus,  Lemma \ref{krational} shows that every group is $\frac{1}{h(G)}$-rational. Note that this definition of $\frac{1}{k}$-rational coincides with the notion of $k$-semi-rational groups introduced in \cite{CD}.

Given $G$ a finite group, we write $\pi(G)$ to denote the set prime divisors of $|G|$.  Gow \cite{G} proved  that $\pi(G)\subseteq \{2,3,5\}$ for any solvable rational group. In the case of semi-rational groups, Theorem 2 of \cite{CD} asserts that $\pi(G)\subseteq \{2,3,5,7,13,17\}$ for any solvable semi-rational group. It is not known whether there exists a semi-rational solvable group whose order is divisible by $17$.

It is worth mentioning that there exists a dual version of these results for characters. We say that a group $G$ is quadratic rational if $|\Q(\chi):\Q|\leq 2$ for any $\chi \in \Irr(G)$. Theorem A of \cite{Tent} shows that $\pi(G)\subseteq \{2,3,5,7,13\}$ for every solvable semi-rational group $G$.

We present the $\frac{1}{k}$-rationality from a different point of view. Given $g \in G$, we define $$\mathbf{B}_G(g)=\mathbf{N}_G(\langle g \rangle )/\mathbf{C}_G(g).$$ We observe that $\mathbf{B}_G(g)$ naturally embeds in $\Aut(\langle g \rangle)$. The following result determines the $\frac{1}{k}$-rationality in terms of $\mathbf{B}_G(g)$ (see Theorem 3.18 of \cite{N18}).

\begin{thm}\label{BG}
Let $G$ be a group and let $g \in G$. Then $$|\Q(g):\Q|=|\Aut(\langle g \rangle):\mathbf{B}_G(g)|.$$ 
\end{thm}

From Theorem \ref{BG}, we deduce that $G$ is $\frac{1}{k}$-rational if and only if $|\mathbf{B}_G(g)|\geq \varphi(|g|)/k$ for any $g \in G$, where $|g|$ denotes the order of $g$ and $\varphi$ denotes  Euler's totient function. Moreover, this happens if and only if there exist $r_1,\ldots,r_k \in \{0,\ldots, o(g)-1\}$ such that, any generator of $\langle g \rangle $ is conjugated of $g^{r_i}$ for some $i$.

\medskip

Before continuing, we recall the structure of Galois groups of cyclotomic fields. Let $n=p^a$ for a prime $p$ and an integer $a\geq 1$. Then 
$$\Gal(\Q_{n}/\Q)\cong \mathsf{C}_{p^{a-1}(p-1)}$$ for $p>2$ and  
$$\Gal(\Q_{2^a}/\Q)\cong \begin{cases}
  1  &  \text{ for }  a=1 \\
  \mathsf{C}_2 &  \text{ for }   a=2\\
  \mathsf{C}_{2^{a-2}}\times \mathsf{C}_2 & \text{ for }   a\geq 3
\end{cases}$$
for $p=2$. Assume now that $n=p_1^{a_1}\ldots p_t^{a_t}$ with $2\leq p_1<p_2\ldots<p_t$, $a\geq 0$ and $a_i\geq 1$. Then
$$\Gal(\Q_{n}/\Q)\cong \Gal(\Q_{p_1^{a_1}}/\Q)\times \cdots \times \Gal(\Q_{p_t^{a_t}}/\Q).$$

Given $G$ a group and a prime $p$, we write $k_p(G)$ to denote the number of conjugacy classes whose elements are $p$-elements. The next  result shows that $k_p(G)$ is bounded in terms of $h(G)$. We also write $n(G)$ to denote the maximum of $k_p(G)$, where $p$ runs through all primes dividing $|G|$. The following is Theorem 1.1 of \cite{CKMS}.

\begin{thm}[Çinarci, Keller, Maróti, Simion]\label{CKMS}
Let $G$ be a group. Then $G$ is $n(G)$-bounded.
\end{thm}

Now, we are ready to prove Theorem \ref{thmA}.

\begin{proof}[Proof of Theorem \ref{thmA}]
We claim that $n(G)\leq 3h(G)^2$ for any group $G$. Let $p$  be a prime divisor of $G$ and let us assume that $\Exp(G)$ divides $p^nm$ with $\gcd(p,m)=1$. It is no loss to assume that $n\geq 3$.

Let $x_1,x_2,\ldots, x_t$ be $t$ pairwise non-conjugated $p$-elements. Thus, for any $1\leq i \leq t$ we have that $\Q\subseteq \Q(x_i)\subseteq \Q_{p^n}$ and $|\Q(x_i):\Q|\leq h(G)$.

Assume first that $p>2$. In this case, $\Gal(\Q_{p^n}/\Q)$ is cyclic and hence, there exists at most $h(G)$ different fields $\F$ such that $\Q\subseteq \F\subseteq \Q_{p^n}$ and $|\F:\Q|\leq h(G)$. Let $\F$ be any of these fields. Then there exists at most $h(G)$ classes whose field of values is $\F$. Thus, $t\leq h(G)^2$.

Assume now that $p=2$. Since we are assuming  that $n\geq 3$, we have $\Gal(\Q_{p^n}/\Q)\cong \mathsf{C}_{2^{a-2}}\times \mathsf{C}_2$ and hence there exist at most $3h(G)$ different fields $\mathbb{F}$ such that $\Q\subseteq \mathbb{F}\subseteq \Q_{p^n}$ and $|\mathbb{F}:\Q|\leq h(G)$. Thus, reasoning as before, we have $t\leq 3h(G)^2$.

Thus, the claim holds. The result follows from Theorem \ref{CKMS}.
\end{proof}

\section{Proof of Theorem \ref{thmB}}\label{SecB}

\subsection{Preliminaries}

In the proof of Theorem \ref{thmB}, we will deal with groups with few rational conjugacy classes. We need to introduce some results on the number of rational conjugacy classes. Let $\F$ be a field extension of $\Q$. We define
$$\Cl_{\F}(G)=\{K\in \Cl(G)\mid \Q(K)\subseteq \F\}$$
and 
$$\Irr_{\F}(G)=\{\chi\in \Irr(G)\mid\Q(\chi)\subseteq \F\}.$$

It is easy to observe that $\chi \in \Irr_{\Q}(G)$ and $K \in \Cl_{\Q}(G)$ if and only if they are fixed points of the respective actions of $\Gal(\Q_{|G|}/\Q)$. It is also easy to observe that $|\Irr_{\Q}(G)|\leq f(G)$ and $|\Cl_{\Q}(G)|\leq h(G)$.

In general, $|\Irr_{\Q}(G)|\not=|\Cl_{\Q}(G)|$. For example, the \texttt{SmallGroup(32,42)} group in GAP \cite{gap} has $10$ rational characters and $8$ rational classes. Moreover, the group \texttt{SmallGroup(32,15)} has $6$ rational characters and $4$ rational classes.  Navarro and Tiep \cite{NT} studied the structure of groups with $|\Irr_{\Q}(G)|\leq 2$ or $|\Cl_{\Q}(G)|\leq 2$. Rossi \cite{Rossi} extended this result by studying the groups with $|\Irr_{\Q}(G)|=3$ or $|\Cl_{\Q}(G)|=3$. Summarizing the results in \cite{NT,Rossi} we have the following theorem.

 \begin{thm}[Navarro and Tiep; Rossi]\label{NumRatChar}
Let $G$ be a finite group. Then
\begin{itemize}
\item [(i)] $|\Cl_{\Q}(G)|=1$ if and only if $|\Irr_{\Q}(G)|=1$ if and only if $|G|$ is odd.
\item [(ii)] $|\Cl_{\Q}(G)|=2$ if and only if $|\Irr_{\Q}(G)|=2$.

\item [(iii)] If $|\Cl_{\Q}(G)|=3$, then $|\Irr_{\Q}(G)|=3$.
\end{itemize}
In particular,  if $|\Cl_{\Q}(G)|\leq 3$, then $|\Irr_{\Q}(G)|=|\Cl_{\Q}(G)|$.
\end{thm}

It is  conjectured  that if $|\Irr_{\Q}(G)|=3$, then $|\Cl_{\Q}(G)|=3$. Theorem C of \cite{Rossi} determines the structure of a possible exception to this conjecture.

It is also worth mentioning that Question 1 of \cite{DAL} asks whether $|\Cl_{\Q}(G)|=|\Irr_{\Q}(G)|$ for groups with $|\Irr_{\Q}(G)|\leq 5$.

The groups  \texttt{SmallGroup(32,42)} and  \texttt{SmallGroup(32,15)} show that the actions of $\Gal(\Q_{|G|}/\Q)$ on $\Irr(G)$ and on $\Cl(G)$ are not permutation isomorphic. However, we have the following result of Brauer's (see Theorem 6.32 of \cite{Isaacscar}),  which will allow us to relate  both actions.

 \begin{thm}[Brauer]\label{Iso1}
 Let $G$ be a group and suppose that the group $A$ acts
on $\Irr(G)$ and on $\Cl(G)$. Assume that the actions are compatible, in the sense that
for every $\alpha \in A$ we have $\chi^{\alpha}(g^{\alpha})=\chi(g)$ for every $g \in G$ and for every $\chi \in \Irr(G)$. Then the number of elements of $\Irr(G)$ fixed by $\alpha$ is equal to the number of elements of $\Cl(G)$ fixed by $\alpha$.
\end{thm}

It is easy to see that the actions of $\Gal(\Q_{|G|}/\Q)$ on $\Irr(G)$ and $\Cl(G)$ satisfy the compatibility condition. We also notice that $\chi \in \Irr_{\mathbb{R}}(G)$ and $K\in \Cl_{\mathbb{R}}(G)$ if and only if they are fixed by $\sigma$, where $\sigma$ denotes the complex conjugation automorphism. Thus, using Theorem \ref{Iso1}, we deduce that  $|\Cl_{\mathbb{R}}(G)|=|\Irr_{\mathbb{R}}(G)|$.

 From Theorem \ref{Iso1}, it is also possible to deduce the following result (see Corollary 6.34 of \cite{Isaacscar}).

\begin{cor}[Brauer]\label{Iso2}
Let $G$ be a group and suppose that the group $A$ acts
on $\Irr(G)$ and on $\Cl(G)$. Assume that the actions are compatible, in the sense that
for every $\alpha \in A$ we have $\chi^{\alpha}(g^{\alpha})=\chi(g)$ for every $g \in G$ and for every $\chi \in \Irr(G)$. Then the number of orbits in the action of $A$ on $\Irr(G)$ is equal to the number of orbits in the action of $A$ on  $\Cl(G)$.
\end{cor}

\subsection{Special cases}

In this subsection, we prove that Theorem \ref{thmB} holds for groups satisfying some special conditions.  We begin by using Lemma \ref{krational} to prove Theorem \ref{thmB} for groups with $h(G)=1$.

\begin{pro}\label{h1}
$h(G)=1$ if and only if $G=1$.
\begin{proof}
Assume that $G$ is a group with $h(G)=1$.  By Lemma \ref{krational}, all conjugacy classes of $G$ are rational. Now, the condition $h(G)=1$ implies that $G$ possesses only one rational conjugacy class and hence, $G$ is a group with a unique conjugacy class. Thus, $G=1$.
\end{proof}
\end{pro}

There are specific situations in which the actions of $\Gal(\Q_{|G|}/\Q)$ on $\Irr(G)$ and on $\Cl(G)$ are permutation isomorphic. We say that $x \in G$  is inverse semi-rational if every generator of $\langle x \rangle$ is conjugated to either $x$ or to $x^{-1}$. A group is said to be inverse semi-rational if all its elements are inverse semi-rational. We remark that the inverse semi-rational groups coincide with the \textbf{cut} groups, which appear in the study of the units of the ring $\mathbb{Z}G$ (see the introduction of \cite{BCJM}).  It is  possible to decide whether a group is inverse semi-rational by looking at the fields of values of its irreducible characters. The following is Proposition 2.1 of \cite{BCJM}.

\begin{pro}\label{CUT}
Let $G$ be a group. Then the following are equivalent:

\begin{itemize}
    \item [(i)] $G$ is inverse semi-rational.

    \item [(ii)] For each $x\in G$, we have that either $\Q(x)=\Q$ or $\Q(x)=\Q(\sqrt{-d})$ for some  $d  \geq 1$.

     \item [(ii)] For each $\chi \in \Irr(G)$, we have that either $\Q(\chi)=\Q$ or $\Q(\chi)=\Q(\sqrt{-d})$ for some  $d \geq 1$.
\end{itemize}

    In particular, inverse semi-rational  groups are both quadratic rational and semi-rational.
\end{pro}

Theorem A of \cite{BCJM} proves that the actions of $\Gal(\Q_{|G|}/\Q)$ on $\Irr(G)$ and on $\Cl(G)$ are isomorphic for inverse semi-rational groups. In particular, $|\Cl_{\Q}(G)|=|\Irr_{\Q}(G)|$ for inverse semi-rational groups. This result can be easily extended to the more general class of uniformly semi-rational groups introduced in \cite{USR}. The uniformity property guarantees  the existence of an element of the Galois group that acts as a derangement on non-rational conjugacy classes and irreducible characters, implying again $|\Cl_{\Q}(G)|=|\Irr_{\Q}(G)|$. The key to prove this fact is the following lemma, which was proved as a claim during the proof of Theorem 3.1 in \cite{BCJM}.

\begin{lem}\label{Iso3}
Let $\Gamma$ be a group acting on two finite sets $X$ and $Y$ satisfying the following conditions:
\begin{itemize}

\item[(i)] Each $\sigma\in \Gamma$ fixes the same number of points in $X$ and $Y$.

\item[(ii)] Each $\Gamma$-orbit on $X$ and $Y$ has size at most $2$.

\end{itemize}
Then the actions of $\Gamma$ on $X$ and $Y$ are permutation isomorphic.
\end{lem}

Let $G$ be any group, let $\Gamma=\Gal(\Q_{|G|}/\Q)$,  $X=\Cl(G)$ and $Y=\Irr(G)$. Condition (i) of Lemma \ref{Iso3} holds by Theorem \ref{Iso1}. Observe also that Condition (ii) of Lemma \ref{Iso3} holds if $G$ is both quadratic rational and semi-rational. With these comments, we have the following.

\begin{cor}\label{QuadSemi}
Let $G$ be a quadratic rational and semi-rational group. Then, the actions of $\Gal(\Q_{|G|}/\Q)$ on $\Irr(G)$ and $\Cl(G)$ are permutation isomorphic.
\end{cor}

Now, we prove the following result.

\begin{thm}\label{Main}
Let $G$ be a group with $|\Irr_{\Q}(G)|=|\Cl_{\Q}(G)|$. Then $G$ is quadratic rational if and only if $G$ is semi-rational. Moreover, in such a case, the actions of $\Gal(\Q_{|G|}/\Q)$ on $\Irr(G)$ and $\Cl(G)$ are permutation isomorphic.
\begin{proof}
Let $\Gamma=\Gal(\Q_{|G|}/\Q)$ and let $t=|\Irr_{\Q}(G)|=|\Cl_{\Q}(G)|$. By Theorem \ref{Iso2}, we have that $\Gamma$ has the same number of orbits on $\Irr(G)$ and on $\Cl(G)$. Let $k$ be  this number and observe that $k\geq t$. Now, if we decompose $\Irr(G)$ and $\Cl(G)$ into $\Gamma$-orbits, then
$$\Irr(G)=\Irr_{\Q}(G)\cup(\bigcup_{i=t+1}^{k}T_i)$$
and
$$\Cl(G)=\Cl_{\Q}(G)\cup(\bigcup_{i=t+1}^{k}L_i),$$
where the $T_i$ and the $L_i$ are the non-trivial $\Gamma$-orbits of its respective actions (in particular, $|T_i|\geq 2$ and $|L_i|\geq2$).

Assume first that $G$ is quadratic rational. In this case, $|T_i|=2$ for all $i$. Since $|\Irr(G)|=|\Cl(G)|$, we have 
$$|\Irr_{\Q}(G)|+2(k-t)=|\Irr(G)|=|\Cl(G)|=|\Cl_{\Q}(G)|+\sum_{i=t+1}^{k}|L_i|$$
Finally, by hypothesis, we have that 
$$\sum_{i=t+1}^{k}|L_i|=2(k-t).$$
Since $|L_i|\geq 2$ for all $i$, this forces $|L_i|=2$ for all $i$, or equivalently, that $G$ is semi-rational.

Now, if $G$ is semi-rational, then with a very similar argument we can deduce that $G$ is quadratic rational.

The second part follows from Corollary \ref{QuadSemi}.
\end{proof}
\end{thm}

We mention that when $|G|$ is odd (equivalently, $|\Irr_{\Q}(G)|=|\Cl_{\Q}(G)|=1$) a strengthened version  of Theorem \ref{Main} was known. By Theorem A of \cite{N}, we have that the number of quadratic character coincides with the number of semi-rational classes of odd order groups. In particular, if $G$ is odd, then $G$ is quadratic rational if and only if it is semi-rational. This was observed in the first paragraph of Section 6 of \cite{Tent}.

Now, we prove Theorem \ref{thmB} for groups with $h(G)=2$.

\begin{thm}\label{h2}
Let $G$ be a finite group. If $h(G)=2$, then $f(G)=2$.
\begin{proof}
Assume now that $h(G)=2$. In this case, $G$ is a group $|\Cl_{\Q}(G)|\leq 2$ and hence, by Theorem \ref{NumRatChar}, we have  $|\Irr_{\Q}(G)|=|\Cl_{\Q}(G)|=2$. In addition, by Lemma \ref{krational}, we have that $|\Q(K):\Q|\leq 2$ for all $K\in \Cl(G)$ and hence, $G$ is a semi-rational group.

Thus, $G$ is a semi-rational group with $|\Irr_{\Q}(G)|=|\Cl_{\Q}(G)|$ and hence, by Theorem \ref{Main}, we have that the actions of  $\Gal(\Q_{|G|}/\Q)$ on $\Irr(G)$ and $\Cl(G)$ are permutation isomorphic. Therefore, $f(G)=h(G)=2$.
\end{proof}
\end{thm}

As a consequence, we can prove Theorem \ref{thmB} for groups with $|\Cl_{\Q}(G)|=1$.

\begin{pro}\label{OddCase}
    Let $G$ be a group with $h(G)\leq 3$ and $|\Cl_{\Q}(G)|=1$. Then the actions of $\Gal(\Q_n/\Q)$ on $\Irr(G)$ and $\Cl(G)$ are permutation isomorphic. In particular, $f(G)=h(G)$.
    \begin{proof}
        Since $|\Cl_{\Q}(G)|=1$, it follows that $|G|$ is odd. By hypothesis, we know that $|\Q(g):\Q|\leq 3$ for any $g \in G$.

    We claim that $|\Q(g):\Q|\leq 2$ for any $g\in G$. Assume for contradiction that there exists $g\in G$ such that $|\Q(g):\Q|=3$.  Since $|\Q(g):\Q|=3$, we have that $|\Aut(\langle g \rangle):\mathbf{B}_G(g)|=3$. Then, $\varphi(|g|)=|\Aut(\langle g \rangle)|=3|\mathbf{B}_G(g)|$. Now, we observe that $|\mathbf{B}_G(g)|$ is odd (since it divides $|G|$, which is odd), and hence $\varphi(|g|)$ is odd. This forces $g=1$, which is impossible. The claim follows.

We conclude that $G$ is a group with $h(G)=2$ and the result follows by Theorem \ref{h2}.
    \end{proof}
\end{pro}

\subsection{The general case}

In this subsection we complete the proof of Theorem \ref{thmB}. By the results in the previous subsection, we may assume that  $h(G)=3$ and $|\Cl_{\Q}(G)|\in \{2,3\}$.

\begin{pro}\label{RatOrders}
    Let $G$ be a group with $|\Cl_{\Q}(G)|\in \{2,3\}$. Then $\{|g|\mid  \Q(g)=\Q\}$ is one of the following
\begin{itemize}
    \item [(i)] $\{1,2\}$.

    \item[(ii)] $\{1,2,4\}$.

    \item[(iii)] $\{1,2,q\}$, where $q$ is an odd prime.
\end{itemize}
\begin{proof}
Since $|\Cl_{\Q}(G)|>1$, we have that $|G|$ is even and hence, there exists $z \in G$ an involution. Then $\{1^G, z^G\}$ are two rational classes. If $|\Cl_{\Q}(G)|=2$, then they are the unique rational classes in $G$ and hence case (i) holds.

Assume now that $|\Cl_{\Q}(G)|=3$. Let $x\in G$ be a rational element such that $x^G\not \in \{1^G, z^G\}$. Thus, the rational conjugacy classes of $G$ are $\{1^G, z^G, x^G\}$. It only remains to determine the order of $x$.

Assume first that there exists $q$ an odd prime dividing $|x|$. Then $|x|=qm$ for an integer $m\geq 1$. We have that $\Q\subseteq \Q(x^m)\subseteq \Q(x)=\Q$ and hence $(x^m)^G$ is a rational class. Since $|x^m|=q>2$, we deduce that $(x^m)^G\not \in \{1^G, z^G\}$. Thus, $(x^m)^G=x^G$, which forces $m=1$ or equivalently, $|x|=q$. Thus, case (iii) holds.

Assume now that $|x|=2^a$ for some $a\geq 1$. Assume that $a\geq 3$. Reasoning as before, we have that $x^2$ is a rational element with $2<|x^2|<|x|$. Thus, $(x^2)^G\not \in \{1^G, z^G, x^G\}$, which is a contradiction. This forces either $|x|=2$ or $|x|=4$. If $|x|=2$, then case (i) holds and if $|x|=4$, then case (ii) holds.
\end{proof}
\end{pro}

\begin{pro}\label{ClassFields}
    Let $G$ be a group with $h(G)\leq 3$ and $|G|$ even.   Then for each  $g\in G$ there exists a prime $p$ such that $\Q(g)\subseteq \Q_{p^3}$.
    \begin{proof}Since $h(G)\leq 3$ and $|G|$ is even, then  one of the cases (i), (ii) or (iii) of Proposition \ref{RatOrders} must hold.  It suffices to study the possible values of $|g|$ in each of the cases.

Assume first that $\{|x|\mid \Q(x)=\Q\}=\{1,2\}$. Let $g \in G$.  If $\Q(g)=\Q$, then the result holds by taking any prime $p$.  Assume that $\Q(g)\not=\Q$.  

Let $1<j$ be a divisor of of $|g|$. Then $\Q\subseteq \Q(g^j)\subseteq \Q(g)$. Since  $|\Q(g):\Q|\in \{2,3\}$, we deduce that either $\Q(g^j)=\Q$ or $\Q(g^j)=\Q(g)$. If $\Q(g^j)=\Q(g)$, then counting the Galois conjugates of $g^G$ and $(g^j)^G$, we have at least $4$ classes whose field of values is $\Q(g)$, so we must have  $\Q(g^j)=\Q$ since $h(G)\leq 3$. Thus, $|g^j|\in \{1,2\}$ for any $j>1$ dividing $|g|$. This forces $|g|=2$, $|g|=4$ or $|g|=p$ for a prime $p$. In the first two cases $\Q(g)\subseteq \Q_8$ and in the third case $\Q(g)\subseteq \Q_p\subseteq \Q_{p^3}$.

Assume now that $\{|x|\mid  \Q(x)=\Q\}=\{1,2,4\}$. In this case, reasoning as before, we have that either $|g|= p$, for a prime $p$, or $|g|\in\{1,2,4,8\}$. In the first case $\Q(g)\subseteq \Q_{p^3}$ and in the second case $\Q(g)\subseteq \Q_8$.

Finally, let us assume that $\{|x|\mid \Q(x)=\Q\}=\{1,2,q\}$ for an odd prime $q$. In this case, we have that either $|g|=p$ for a prime $p$, or $|g|\in\{1,2,4,2q,q^2\}$. Since $\Q_{2q}=\Q_q$, we have that if $|x|\in \{2q,q^2\}$, then $\Q(g)\subseteq \Q_{q^3}$.
    \end{proof}
\end{pro}

Now, we prove Theorem \ref{thmB}.

\begin{thm}
    Let $G$ be a group with $h(G)\leq 3$, then $h(G)=f(G)$.
    \begin{proof}
Since $h(G)$, we have that $|\Cl_{\Q}(G)|\leq 3$ and that $|\Cl_{\Q}(G)|=|\Irr_{\Q}(G)|$ by Theorem \ref{NumRatChar}. If $|\Cl_{\Q}(G)|=1$, then the result follows by Proposition \ref{OddCase}. Let us assume that $|\Cl_{\Q}(G)| \in \{2,3\}$.

By Proposition \ref{ClassFields}, we have that $\Q(G)\subseteq \Q_n$ for some $n=2^3p_{2}^{3}\cdots p_{t}^3$ for $2=p_1<p_2<\ldots <p_t$ primes. If $i>2$, then $p_i$ is odd and hence $\Gal(\Q_{p_{i}^3}/\Q)=\langle \sigma_i \rangle$ for an automorphism $\sigma_i$ with $|\sigma_i|=(p_i-1)p_i^2$. In the case $p_1=2$, we have that $\Gal(\Q_8/\Q)=\langle \sigma_0 \rangle\times \langle \sigma_1 \rangle$ with $|\sigma_0|=|\sigma_1|=2$. 

Let $K\in \Cl(G)$ with $\Q(K)\not=\Q$. We aim is to prove that $$|\{T \in \Cl(G)\mid \Q(T)=\Q(K)\}|=|\{\chi\in \Irr(G)\mid\Q(\chi)=\Q(K)\}|.$$ By Proposition \ref{ClassFields}, we have that there exists $i \in \{1,\ldots, t\}$ such that $\Q(K)\subseteq \Q_{p_i^3}$. In any case, we have that $\Gal(\Q_{p_i^3}/\Q(K))=\langle \tau \rangle$, where $\tau$ does not fix any element in $\Q_{p_i^3}\setminus \Q(K)$.

Assume first that $i=1$, that is $\Q(K)\subseteq \Q_8$. Let us set $$\sigma=(\tau, \sigma_2,\ldots,\sigma_t)\in \Gal(\Q_n/\Q)$$
By Brauer's Theorem, we have that
$$|\{T \in \Cl(G)\mid T^{\sigma}=T\}|=|\{\chi \in \Irr(G)\mid \chi^{\sigma}=\chi\}|.$$

Let $\chi \in \Irr(G)$ and $g \in G$. We claim that $\chi(g)^{\sigma}=\chi(g)$ if and only if $\chi(g)\in \Q(K)$. To see this, first recall that exists $j$ such that $\chi(g)\in \Q(g)\subseteq \Q_{p_j^3}$. If $j\geq 2$, then $\chi(g)=\chi(g)^{\sigma}=\chi(g)^{\sigma_j}$, which forces $\chi(g)\in \Q$. If $j=1$, then $\chi(g)\in \Q_8$ and $\chi(g)=\chi(g)^{\sigma}=\chi(g)^{\tau}$ and hence $\chi(g)\in \Q(K)$. Thus, $\chi(g)^{\sigma}=\chi(g)$ if and only if $\chi(g)\in \Q(K)$. 

Thus, 
$$\{T \in \Cl(G)\mid T^{\sigma}=T\}=\Cl_{\Q}(G)\cup \{T \in \Cl(G)\mid \Q(T)=\Q(K)\}$$
and 
$$\{\chi \in \Irr(G)\mid \chi^{\sigma}=\chi\}=\Irr_{\Q}(G)\cup \{\chi \in \Irr(G)\mid \Q(\chi)= \Q(K)\}.$$
        Therefore, applying Theorem \ref{Iso1}, and the fact that $|\Cl_{\Q}(G)|=|\Irr_{\Q}(G)|$, we deduce that $|\{T \in \Cl(G)\mid \Q(T)=\Q(K)\}|=|\{\chi\in \Irr(G)\mid \Q(\chi)=\Q(K)\}|$.

        Thus, the result holds when $\Q(K)\subseteq \Q_8$. In particular, we deduce that $|\Irr_{\Q_8}(G)|=|\Cl_{\Q_8}(G)|.$

\bigskip

        Now, let us assume that $\Q(K)\subseteq \Q_{p_i^{3}}$ for some $i \geq 2$. Let us set 
        $$\sigma=(1,\sigma_2, \ldots, \sigma_{i-1},\tau, \sigma_{i+1},\ldots, \sigma_t).$$
        Applying Brauer's Theorem again, we have that 
        $$|\{T \in \Cl(G)\mid T^{\sigma}=T\}|=|\{\chi \in \Irr(G)\mid \chi^{\sigma}=\chi\}|.$$
Reasoning as before,  we see that if $\chi \in \Irr(G)$ and $g \in G$, then  $\chi(g)^{\sigma}=\chi(g)$ if and only if $\chi(g)\in \Q(K)$ or $\chi(g)\in \Q_8$. Moreover, 
 by Proposition \ref{ClassFields}, we have that either $\Q(T)= \Q(K)$ or $\Q(T)\subseteq \Q_8$ for any $T \in \Cl(G)$ satisfying $T^{\sigma}=T$. Thus, 
$$\{T \in \Cl(G)\mid T^{\sigma}=T\}=\Cl_{\Q_8}(G)\cup \{T \in \Cl(G)\mid \Q(T)=\Q(K)\}$$
and 
$$\{\chi \in \Irr(G)\mid \chi^{\sigma}=\chi\}=\Irr_{\Q_8}(G)\cup \{\chi \in \Irr(G)\mid \Q(K)\subseteq \Q(\chi)\}.$$

Since $|\Irr_{\Q_8}(G)|=|\Cl_{\Q_8}(G)|$, and $|\Gal(\Q(K)/\Q)|=|\{T \in \Cl(G)\mid \Q(T)=\Q(K)\}|$ we deduce that
$$|\Gal(\Q(K)/\Q)|=|\{\chi \in \Irr(G)\mid \Q(K)\subseteq \Q(\chi)\}|.$$
Let $\psi \in \Irr(G)$ with $\Q(K)\subseteq \Q(\psi)$. If $\gamma\in \Gal(\Q(\psi)/\Q)$, then $\Q(K)\subseteq \Q(\psi^{\gamma})$ and hence 
$$|\Gal(\Q(K)/\Q)|=|\{\chi \in \Irr(G)\mid \Q(K)\subseteq \Q(\chi)\}|\geq |\Gal(\Q(\psi)/\Q)|.$$
This forces $\Q(\psi)=\Q(K)$ and the set $\{\chi \in \Irr(G)\mid \Q(K)\subseteq \Q(\chi)\}$ is just the set of Galois conjugates of $\psi$. As a consequence $$|\{T \in \Cl(G)\mid \Q(T)=\Q(K)\}|=|\{\chi\in \Irr(G)\mid \Q(\chi)=\Q(K)\}|$$
and the result follows.
    \end{proof}
\end{thm}

 We observe that there exist examples of groups with $h(G)\neq f(G)$. However, it is not clear whether Theorem \ref{thmB} is best possible. Inspecting the groups of order at most $128$  (using the \texttt{SmallGroup} database in GAP \cite{gap}), we have  not found an example of group with  $h(G)\neq f(G)$ such that either $f(G)\leq 5$ or $h(G)\leq 5$.  Thus, we put forward the following conjecture.

\begin{con}\label{HGConjecture}
    Let $G$ be a group. If $\min\{f(G),h(G)\}\leq 5$, then $f(G)=h(G)$.
\end{con}

\bigskip

\centerline{\bf Acknowledgement}

\bigskip

This work is part of the first author's thesis under the supervision of A. Moretó. He would like to thank him. He would also like to thank T.\,C. Burness, G. Malle, N. Rizo and G.\,A\,.L Souza for their many comments and helpful suggestions.  Parts of this work were done when the second author was visiting the  University of València. He thanks the CARGRUPS research team at the University of València for their kind hospitality. The second author wants to thank E. Pacifici and Á. del Río for their valuable feedback and insights on the manuscript.

\end{document}